\documentclass[10pt,oneside]{amsart}
\usepackage
[
	left=1.3in, right=1.3in, top=1in, bottom=1in
]
{geometry}
\title{Anosov geodesic flows on surfaces}
\author{Hao-Tong Yan \\ \scriptsize The Ohio State University}
\usepackage{geometry}
\usepackage{amsmath,amssymb,amsthm}
\usepackage{setspace}
\usepackage{palatino}
\usepackage{float}
\usepackage{todonotes}
\usepackage{verbatim}
\usepackage{tikz}
\usepackage{mathrsfs}
\usepackage[hidelinks]{hyperref}
\usepackage{todonotes}

\newtheorem{theorem}{Theorem}[section]
\newtheorem{lemma}[theorem]{Lemma}
\newtheorem{proposition}[theorem]{Proposition}
\newtheorem{corollary}[theorem]{Corollary}
\theoremstyle{definition}
\newtheorem{definition}[theorem]{Definition}
\newtheorem{remark}[theorem]{Remark}

\theoremstyle{remark}

\renewcommand{\a}{\alpha}
\renewcommand{\b}{\beta}
\newcommand{\g}{\gamma}
\renewcommand{\d}{\delta}

\renewcommand\epsilon\varepsilon
\newcommand{\e}{\epsilon}

\renewcommand{\k}{\kappa}

\newcommand{\s}{\sigma}

\newcommand{\vp}{\varphi}

\newcommand{\maps}{\longmapsto}

\renewcommand{\leq}{\leqslant}
\renewcommand{\geq}{\geqslant}
\renewcommand{\sb}{\subseteq}

\newcommand{\cd}{\cdot}

\newcommand{\abs}[1]{\left\lvert #1 \right\rvert}
\newcommand{\magn}[1]{\left\lVert #1 \right\rVert}
\newcommand{\dd}[2]{\frac{d #1}{d #2}}
\newcommand{\pd}[2]{\frac{\partial #1}{\partial #2}}
\newcommand{\bb}[1]{\mathbb{#1}}

\newcommand{\f}[2]{\frac{#1}{#2}}

\newcommand{\brak}[1]{\left\langle #1\right\rangle}
\newcommand{\set}[1]{\left\{#1\right\}}

\renewcommand{\bar}[1]{\overline{#1}}

\begin{document}
\maketitle

\section{Introduction}
This paper is an exposition of the major results of P. Eberlein's paper, \emph{When is a geodesic flow of Anosov type? I} \cite{Eberlein}, in the special case when the manifold $M$ is a surface. We follow Eberlein's coverage closely, adding details when helpful, and taking advantage of simplifications given by the dimension two case. The objective is to give readers a more tractable introduction to the important topic of Anosov geodesic flows, which highlights key concepts and arguments without worrying about generalizations to arbitrary dimension. As such, this paper is suited for a beginning graduate or advanced undergraduate student who is familiar with basic Riemannian geometry, such as that covered by the first half of \cite{doCarmo}.

The main arc of this paper is divided into three parts. Section 2 highlights some basic concepts and results from Riemannian geometry. For more in-depth coverage, refer to \cite{doCarmo}. Section 3 details useful intermediary results when the manifold $M$ has no conjugate points. The work in these two sections culminate in the statement and proof of the main result in Section 4, which defines an Anosov flow and specifies conditions on Jacobi fields and curvature that are equivalent to the geodesic flow being of Anosov type.

The author would like to thank Andrey Gogolev (The Ohio State University) for advising this project and Jingyin Huang (The Ohio State University) for providing helpful feedback.
\section{Preliminaries}
We start with basic definitions and results. While our main theorem assumes $M$ is a complete $C^\infty$ Riemannian manifold of dimension two, the results in this section are dimension agnostic; hence for now we may assume $M$ is a complete $C^\infty$ Riemannian manifold of dimension $n \geq 2$ with Riemannian metric $\brak{,}$. In all cases, $M$ is equipped with the Levi-Civita connection. Let $TM$ and $SM$ denote the full tangent bundle and unit tangent bundle of $M$, respectively, and let $\pi$ be the projection map onto $M$ in either case. For any vector $v \in TM$, let $\g_v$ be the unique geodesic in $M$ such that $\g_v'(0) = v$, which is guaranteed by completeness of $M$.

A \emph{complete flow} $t \to \vp^t: N \to N$ on a $C^\infty$ manifold $N$ is a homeomorphism of $(\bb R, +)$ into the group of homeomorphisms of $N$. We say the flow is $C^k$ differentiable, $0\leq k\leq \infty$, if the map $(t, n) \to \vp^t(n): \bb R\times N\to N$ is $C^k$ differentiable.

\begin{definition}\label{GeodesicFlow}
	Let $t\in\bb R$ and let $v\in TM$. We define the \emph{geodesic flow} $g^t: TM \to TM$ on $TM$ as
\[
	g^tv = \g_v'(t) \;,
\]
	which is the velocity of $\g_v$ at time $t$.
\end{definition}
	
	The geodesic flow is a complete $C^\infty$ flow on $TM$, and also in $SM$, since $g^t$ leaves $SM$ invariant for all $t\in\bb R$. Let $Y$ denote the vector field in both $TM$ and $SM$ induced by the geodesic flow. In other words, for any real-valued function $f: TM\to\bb R$ (or $f: SM\to\bb R$), we have
\[
	Y(v)f = \left. \dd {}t f(g^tv)\right\rvert_{t=0} \; .
\]
	Working with this definition, we have for any $s\in\bb R$,
\[
	\left(dg^s Y(v)\right) f = \left.\dd {}t f(g^sg^tv)\right\rvert_{t=0} = \left.\dd {}t f(g^t g^sv)\right\rvert_{t=0} = Y(g^sv) f \; ,
\]
	so $dg^s Y(v) = Y(g^s v)$.

\begin{definition}
	A vector field $J$ along a geodesic $\g$ of $M$ is called a \emph{Jacobi field} if
\[
	J''(t) + R(\g'(t), J(t))\g'(t) = 0 \; ,
\]
	where the ticks over $J$ denote covariant differentiation along $\g$ and $R$ is the curvature tensor.
\end{definition}
	
	By passing to local coordinate systems along points of $\g$, we see that this equation yields a system of $n$ linear ODEs of the second order, so $J$ is uniquely determined by the values $J(0)$ and $J'(0)$. Therefore, we have $2n$ linearly independent Jacobi fields along $\g$; let $J(\g)$ denote their span. Let $J_0(\g)\sb J(\g)$ be the $(2n-2)$-dimensional subspace of \emph{perpendicular} Jacobi fields $J$ along $\g$, that is, where $\brak{J(t),\g'(t)} = 0$ for all $t\in\bb R$. The \emph{tangential} Jacobi fields $J$ along $\g$ can be expressed in the form $J(t) = (\a + \b t)\g'(t)$, where $\a,\b \in \bb R$.

\begin{definition}
	For $a\neq b$, we say that $\g(a)$ and $\g(b)$ are \emph{conjugate} if there exists a nonzero Jacobi field $J$ along $\g$ such that $J(a) = J(b) = 0$. We say that $M$ has \emph{no conjugate points} if there exists no pair of conjugate points along any geodesic in $M$, or in other words, for any geodesic $\g$ in $M$ and any nonzero Jacobi field $J$ along $\g$, $J(t) = 0$ for at most one number $t\in\bb R$.
\end{definition}
Let $J$ be a Jacobi field along a geodesic $\g$. Recall that if $a \neq b$ and $\g(a)$ and $\g(b)$ are not conjugate, then $J$ is uniquely determined by $J(a)$ and $J(b)$. Further, if $J$ is perpendicular to $\g$ at any two nonconjugate points, then it is perpendicular everywhere.

Let $N \sb M$ be a proper $C^\infty$ Riemannian submanifold of $M$, and $\g$ be a unit speed geodesic of $M$ such that $\g'(0)$ is perpendicular to $T_{\g(0)}N$. A Jacobi field $J$ along $\g$ is an \emph{$N$-Jacobi field} if $J$ is perpendicular to $\g$, $J(0)\in T_{\g(0)}N$, and $J'(0) + S_{\g'(0)}(J(0)) \in (T_{\g(0)}N)^\perp$, where $S_{\g'(0)}$ is the linear operator given by the second fundamental form on $N$. Equivalently, $J$ is an $N$-Jacobi field if it can be written as
\[
	J(s) = \pd rt (s, 0) \; ,
\]
	where $r: (-\infty, \infty) \times (-\e, \e) \to M$ is a $C^\infty$ variation of the form
\[
	r(s,t) = \exp_{\pi Z(t)} (sZ(t)) = (\pi\circ g^s)(Z(t))\; ,
\]
	where $Z(t)$ is a $C^\infty$ curve in the unit normal bundle of $N$ such that $Z(0) = \g'(0)$. Note that $\partial r/\partial t$ denotes shorthand for $dr (\partial/\partial t)$.

If $N$ is a point, then a Jacobi field $J$ is an $N$-Jacobi field if and only if $J(0) = 0$ and $J$ is perpendicular to $\g$. If $N$ is a geodesic $\s$ and $J(0)\neq 0$, then a Jacobi field $J$ is an $N$-Jacobi field if and only if $J(0)$ is tangent to $\s$, $J$ is perpendicular to $\g$, and $\brak{J(0), J'(0)} = 0$.

\begin{definition}
	For an arbitrary proper submanifold $N\sb M$ and a perpendicular unit speed geodesic $\g$, the point $\g(a)$, $a\neq 0$, is a \emph{focal point} of $N$ along $\g$ if there exists a nontrivial $N$-Jacobi field $J$ along $\g$ such that $J(a) = 0$. $M$ is said to have \emph{no focal points} if no geodesic $N = \s$ has focal points along any unit speed geodesic perpendicular to $\s$.
\end{definition}

\begin{proposition}
	$M$ has no focal points if and only if the following condition is satisfied: For any unit speed geodesic $\g$ in $M$ and any Jacobi field $J$ on $\g$ such that $J(0) = 0$ and $J'(0)\neq 0$, we have
\[
	\dd{}t \brak{J(t), J(t)}  > 0
\]
	for all $t > 0$.
\end{proposition}
\begin{proof}
	We follow the proof from \cite{OSullivan}. Since $J(0) = 0$, we can write $J(t) = J_1(t) + \a t\g'(t)$, where $J_1$ is a perpendicular Jacobi field with $J_1(0) = 0$ and  $\a\in\bb R$. Then
\[
	\dd {}t \brak{J(t), J(t)} =  \dd {}t \brak{J_1(t), J_1(t)} + 2\a^2 t \; ,
\]
	so it suffices to show the result for nonzero perpendicular Jacobi fields vanishing at $t=0$.
	
	Let $J$ be such a Jacobi field, suppose $M$ has no focal points, and suppose there exists $t_0 > 0$ such that $\dd {}t \brak{J(t), J(t)}\rvert_{t=t_0} = 0$. Then $\brak{J'(t_0), J(t_0)} = 0$ and by picking the smallest $t_0$ with this property, we can ensure that $J(t_0) \neq 0$. Let $\s$ be the geodesic with $\s'(0) = J(t_0)$, and let $L = J(t_0 - t)$, so $L$ is Jacobi field along the geodesic $\g(t_0-t)$. Thus $\brak{L'(0), L(0)} = 0$, and by the above discussion on focal points, $L$ is an $N$-Jacobi field with $N = \s$. But $L(t_0) = 0$, so $\g(t_0)$ is a focal point of $\s$ along $\g(t_0-t)$, a contradiction. Since $\dd {} t \brak{J(t),J(t)} > 0$ for sufficiently small $t>0$, this contradiction shows that $\dd {} t \brak{J(t),J(t)} > 0$ for all $t>0$.
	
	Conversely, suppose $M$ has focal points. Then there exist geodesics $\g$ and $\s$ with $\g(0)$ in $\s$ and $\g'(0)$ perpendicular to $\s$ and a non-trivial $N$-Jacobi field $J$ with $N = \s$ such that $J(t_0) = 0$ for some $t_0 > 0$. Then $L(t) = J(t_0-t)$ is a non-trivial perpendicular Jacobi field with $L(0) = 0$. But $\dd {}t \brak{L,L}\rvert_{t=t_0} = -2 \brak{J'(0), J(0)} = 0$ since $J$ is an $N$-Jacobi field.
\end{proof}

This immediately yields the following result, which will be useful later.
\begin{corollary}\label{HaoNoFocal}
	Suppose $M$ has no focal points and let $\g$ be a geodesic. Then for $t \geq s\geq 0$ and any Jacobi field $J$ along $\g$ with $J(0) = 0$ and $J'(0) \neq 0$, we have
\[
	\magn{J(t)} \geq \magn{J(s)} \; .
\]
\end{corollary}

We want to define a natural isomorphism between $T_v(TM)$ and $J(\g_v)$ for any complete $C^\infty$ Riemannian manifold $M$ and any $v\in TM$. For each $v\in TM$, $d\pi: T_v(TM) \to T_{\pi v} M$ is linear and the kernel of $d\pi$ is the $n$-dimensional \emph{vertical subspace} of $T_v(TM)$. We define a \emph{connection map} $K: T(TM) \to TM$ such that for each $v\in TM$, $K: T_v(TM) \to T_{\pi v}M$ is linear. The kernel of $K$ is the $n$-dimensional \emph{horizontal subspace} of $T_v(TM)$, and the intersection of the horizontal and vertical subspaces is the zero vector.

Given a vector $\xi \in T_v(TM)$, let $Z: (-\e, \e) \to TM$ be a $C^\infty$ curve with initial velocity $\xi$, and let $\a = \pi\circ Z: (-\e, \e) \to M$. We define $K(\xi) = Z'(0) \in T_{\pi v} M$, where $Z'(0)$ is the covariant derivative of $Z$ along $\a$ evaluated at $t=0$. By passing to a local coordinate system about $v$, we can show that $K(\xi)$ does not depend on the chosen curve $Z$. Note that for $\xi \in T_v(SM)$, we can choose the curve $Z$ to lie in $SM$.

We define a natural metric on $TM$ called the \emph{Sasaki metric}, which has the property that the horizontal and vertical subspaces of $T_v(TM)$ are orthogonal. Given vectors $\xi, \eta \in T_v(TM)$, we define
\[
	\brak{\xi, \eta}_v = \brak{d\pi\xi, d\pi\eta}_{\pi v} + \brak{K\xi, K\eta}_{\pi v} \; .
\]
We can also define this metric on $SM$ in the same way, since $d\pi|_{T(SM)}$ and $K|_{T(SM)}$ both map $T(SM)$ to $SM$. From now on, this will be the metric we use whenever dealing with $TM$ or $SM$.

\begin{proposition}\label{P1.5}
	Let $p: N\to M$ be a surjective local isometry of complete Riemannian manifolds. Then
\begin{enumerate}
	\item $P = dp: TN \to TM$ is a surjective local isometry carrying $SN$ onto $SM$,
	\item $\magn{dg^t\xi} = \magn{dg^t dP(\xi)}$ for any $t\in\bb R$ and any $\xi\in T(TN)$, where $g^t$ denotes the geodesic flow on both $TN$ and $TM$, and
	\item $dPY(v) = Y(Pv)$ for any $v\in TN$, where $Y$ denotes the vector field defined by $g^t$ in both $TN$ and $TM$.
\end{enumerate}
\end{proposition}
\begin{proof}
	Define the projection maps $\pi_1: TN\to N$, $\pi_2: TM\to M$, and the connection maps $K_1: T(TN)\to TN$ and $K_2: T(TM) \to TM$. We have
\begin{enumerate}
	\item[(i)] $p\circ \pi_1 = \pi_2\circ P$
	\item[(ii)] $P\circ d\pi_1 = d\pi_2\circ dp$
	\item[(iii)] $P\circ K_1 = K_2\circ dP$, and
	\item[(iv)] $P\circ g^t = g^t \circ P$ for any $t\in\bb R$,
\end{enumerate}
	which can be verified from the definitions and by passing to local coordinates. Relations (ii) and (iii) show that $P$ is a local isometry, surjectivity of $p$ guarantees surjectivity of $P$, and the fact that $P(SN) = SM$ follows from the fact that $p$ is a local isometry. Assertion (2) follows from (iv) and (1), since $\magn{dg^t dP(\xi)} = \magn{dP\circ dg^t(\xi)} = \magn{dg^t\xi}$. Assertion (3) follows from (iv) and working with the explicit definition of $Y(v)$ defined after Definition \ref{GeodesicFlow}.
\end{proof}

\begin{definition}\label{D1.6}
	For any $v\in TM$ and any $\xi\in T_v(TM)$, let $J_\xi$ be the unique Jacobi field on $\g_v$ such that $J_\xi(0) = d\pi\xi$ and $J_\xi'(0) = K\xi$.
\end{definition}

We recall Lemma 3.4 from \cite{doCarmo}, which is a symmetry result used to prove Gauss's Lemma.

\begin{lemma}
	Let $A\sb\bb R^2$ be a (possibly unbounded) open rectangle. Let $r(u,v): A\to M$ be differentiable. Then
\[
	\f D{\partial t} \pd rs (s,t) = \f D{\partial s} \pd rt (s,t) \; ,
\]
	where $D/\partial t$ denotes covariant differentiation with respect to $t$ along the curve $t \to r(s, t)$ for fixed $s$, and $D/\partial s$ denotes covariant differentiation with respect to $s$ along the curve $s \to r(s, t)$ for fixed $t$.
\end{lemma}

Let $v\in SM$ and let $r(s,t) = \exp_{\pi Z(t)} (sZ(t)) = (\pi\circ g^s)(Z(t))$ be the variation of $\g_v$ defined in the discussion of focal points above, with the difference that now $Z$ is any curve in $TM$ with initial velocity $\xi\in T_v(TM)$. Then for any $s\in\bb R$, we claim that
\[
	J_\xi(s) = \pd rt (s, 0) \; .
\]
	It suffices to show that $J_\xi(0) = \pd rt(0,0)$ and $J_\xi'(0) = \f D{\partial s} \pd rt(0,0)$. By direct computation, we have $\pd rt(s,0) = d\pi \circ dg^s \xi$, so $\pd rt(0,0) = d\pi \xi = J_\xi(0)$. Next, notice that $\pd rs (s, t) = \g_{Z(t)}' (s) = g^s (Z(t))$, so the preceding lemma gives us that
\[
	\f D{\partial s} \pd rt(s, t) = \f D{\partial t} \pd rs (s, t) = \f D{\partial t}g^s (Z(t))\; ,
\]
	which gives us
\[
	\f D{\partial s} \pd rt(0, 0) = Z'(0) = K \xi = J_\xi'(0)\; ,
\]
	as desired.
	
	If $\xi\in T_v(SM)$, we may choose the curve $Z$ to lie in $SM$; the $u$-parameter curves of the variation are then unit speed geodesics of $M$. We claim that the map $s \to \brak{J_\xi (s), \g_v'(s)}$ is a constant function if and only if $\xi\in T_v(SM)$. Indeed, consider
\[
	\dd {}s \brak{J_\xi (s), \g_v'(s)} = \brak{J_\xi'(s), \g_v'(s)} + \brak{J_\xi(s), \f {D\g_v'}{ds} (s)} = \brak{J_\xi'(s), \g_v'(s)} 
\]
	since $\g_v$ is a geodesic. We have
\begin{align*}
	\brak{J_\xi'(s), \g_v'(s)} &= \brak{\f D{\partial s} \pd rt (s,0) , \pd rs (s,0)} \\
	&= \brak{\f D{\partial t} \pd rs (s,0) , \pd rs (s,0)} \\
	&= \f 12 \pd {}t \left.\brak{\pd rs (s,t) , \pd rs (s,t)}\right\rvert_{t=0} \\
	&= \f 12 \pd {}t \left.\brak{\g_{Z(t)}'(s) , \g_{Z(t)}'(s)}\right\rvert_{t=0} \\
	&= \f 12 \dd {}t \left.(\magn{Z(t)}^2)\right\rvert_{t=0} \; ,
\end{align*}
	which is zero if and only if we can choose $Z(t)$ to be a curve in $SM$, which is true if and only if $\xi\in T_v(SM)$. The following result is now straightforward to verify. For (5) we also need the result that $KY(v) = 0$ and $d\pi Y(v) = v$ for all $v\in TM$, which follows from the definitions.

\begin{proposition}\label{P1.7}
	Let $v\in TM$. Then
\begin{enumerate}
	\item $\xi\to J_\xi$ is a linear isomorphism of $T_v(TM)$ onto $J(\g_v)$,
	\item $J_\xi(t) = d\pi\circ dg^t(\xi)$ and $J_\xi'(t) = K\circ dg^t(\xi)$ for $t\in\bb R$,
	\item $\xi\in T_v(TM)$ lies in $T_v(SM)$ for $v\in SM$ if and only if $\brak{K\circ dg^t(\xi),g^t v} = \brak{J_\xi'(t), \g_v'(t)} = 0$ for all $t\in\bb R$, if and only if $t\to \brak{J_\xi(t), \g_v'(t)}$ is a constant function,
	\item $\brak{\xi, Y(v)} = 0$ for $v\in SM$ and $\xi\in T_v(SM)$ if and only if $\brak{J_\xi(t), \g_v'(t)} = 0$ for all $t\in \bb R$, where $Y$ is the flow vector field, and
	\item $\brak{dg^t\xi, Y(dg^t v)} = 0$ for all $t\in\bb R$, if $v\in SM$ and $\xi\in T_v(SM)$ satisfy $\brak{\xi, Y(v)} = 0$.
\end{enumerate}
\end{proposition}

\begin{remark}\label{R1.8}
	It follows from (2) above that for any $t\in\bb R$, any $v\in TM$ and any $\xi\in T_v(TM)$ that $\magn{dg^t\xi}^2 = \magn{J_\xi(t)}^2 + \magn{J_\xi'(t)}^2$.
\end{remark}

\begin{definition}\label{D1.9}
	Let $\g_n$ be a sequence of geodesics in $M$, and let $J_n$ be a sequence of Jacobi fields such that $J_n$ is defined on $\g_n$ for every $n$. If $v_n = \g_n'(0)$, choose $\xi_n\in T_{v_n}(TM)$ so that $J_n = J_{\xi_n}$. We say that the Jacobi fields $J_n$ converge to a Jacobi field $J$ on a geodesic $\g$ if $\xi_n \to \xi$ in $T(TM)$, where $v = \g'(0)$, $\xi\in T_v(TM)$, and $J = J_\xi$.
\end{definition}

Note that $\xi_n\to\xi$ in $T(TM)$ if and only if $K\xi_n\to K\xi$ and $d\pi\xi_n \to d\pi\xi$. Therefore $J_n$ converges to $J$ if and only if $\g_n'(0) \to \g'(0)$, $J_n(0)\to J(0)$, and $J_n'(0)\to J'(0)$. If $J_n\to J$ and $(u_n)\sb\bb R$ is a sequence converging to a finite number $u$, then by (2) of Proposition \ref{P1.7}, $J_n(u_n) \to J(u)$ and $J_n'(u_n) \to J'(u)$.

\begin{definition}\label{D1.10}
	$M$ is said to be \emph{compactly homogenous} if there exists a compact set $B\sb M$ such that $M$ is the union of all translates of $B$ by the isometries of $M$.
\end{definition}

If $M$ is homogenous or a Galois Riemannian covering of a compact Riemannian manifold, then $M$ is compactly homogenous. Since isometries preserve sectional curvature, all values of the sectional curvature of $M$ (or Gaussian curvature if $M$ is a surface) are taken at points of $B$, and therefore the sectional curvature of $M$ is uniformly bounded above and below.

\begin{proposition}\label{P1.11}
	Let $M$ be compactly homogenous. For each integer $n > 0$, let $J_n$ be a Jacobi field on a geodesic $\g_n$ with initial velocity $v_n$. If each of the sequences $\magn{v_n}$, $\magn{J_n(0)}$, and $\magn{J_n'(0)}$ is uniformly bounded above, then we can find a sequence $\vp_n$ of isometries of $M$ and a Jacobi field $L$ on a geodesic $\s$, such that by passing to a subsequence, $L_n = d\vp_n J_n \to L$ in the sense of Definition \ref{D1.9}.
\end{proposition}
\begin{proof}
	Let $p_n = \pi v_n$ and choose a sequence $\vp_n$ of isometries of $M$ so that the sequence $q_n = \vp_np_n$ is contained in a compact subset of $M$. $L_n = d\vp_n J_n$ is a Jacobi field on $\s_n = \vp_n\circ \g_n$, and $L_n'(0) = d\vp_n J_n'(0)$. If $w_n = \s_n'(0) = d\vp_nv_n$, then the sequences $\magn{w_n}$, $\magn{L_n(0)}$, and $\magn{L_n'(0)}$ are uniformly bounded above, since $d\vp_n$ is an isometry. Passing to a subsequence, let $q_n\to q$ in $M$ and let $U$ be a local coordinate system about $q$. For each $n$, the three vectors $w_n$, $L_n(0)$, and $L_n'(0)$ lie in $T_{q_n} M$, and relative to the induced coordinate system $\pi^{-1}(U)$ in $TM$, it is easy to see that the coordinates of each of the sequences $w_n$, $L_n(0)$, and $L_n'(0)$ are uniformly bounded in absolute value, so they are contained in a compact subset of $TM$. Passing to a further subsequence, let $w_n\to w$, $L_n(0)\to u$, and $L_n'(0)\to v$, where $u$, $v$, and $w$ are vectors in $T_q M$. Let $\s = \g_w$, and let $L$ be the Jacobi field on $\s$ such that $L(0) = u$ and $L'(0) = v$. Then $L_n\to L$ by the discussion following Definition \ref{D1.9}. 
\end{proof}

If $\vp$ is an isometry of $M$, then $T_\vp = d\vp$ is an isometry of $TM$, which leaves $SM$ invariant. The following result is equivalent to that just proved.
\begin{proposition}\label{P1.12}
	Let $M$ be compactly homogenous. Let $(v_n)\sb TM$ and $(\xi_n)\sb T(TM)$ be sequences such that $\xi_n\in T_{v_n}(TM)$ for every $n$, and that the sequences $\magn{v_n}$ and $\magn{\xi_n}$ are each uniformly bounded above. Then there exists a sequence $\vp_n$ of isometries of $M$ such that by passing to a subsequence,\linebreak $\xi_n^* = dT_{\vp_n} \xi_n$ converges to a vector $\xi^*\in T(TM)$.
\end{proposition}

\newpage
\section{Surfaces without conjugate points}
From here on, assume $M$ is a \emph{surface}. In this section we continue building our framework in order to prove the main result in \textsection 4, and we focus on the case where $M$ has no conjugate points. In many cases, we consider the unit tangent bundle $SM$ instead of $TM$. Assume all geodesics are unit speed. The main goal of this section is to define the $1$-dimensional subspaces $X_s(v)$ and $X_u(v)$ of $T_v(SM)$, which, as we will eventually show, coincide precisely with the subspaces $X_s^*(v)$ and $X_u^*(v)$ in the definition of an Anosov flow when certain conditions are met. We will also state and prove several propositions relevant to these subspaces that will be useful in the final section.

We start by defining a useful method of working with the tangent space along a geodesic. Let $\g$ be a geodesic, and let $E_1(s)$, $E_2(s)$ be a parallel orthonormal frame along $\g$ with $E_2(s) = \g'(s)$; indeed, $E_1(s)$ can be defined by parallel transport. Then a perpendicular vector field $J$ along $\g$ can be written as $J(s) = y(s) E_1(s)$, and we can identify $J$ with the ``curve`` $s \maps y(s)$ in $\bb R$. Note that $J'(s) = y'(s) E_1(s)$. We call such a frame an \textit{adapted frame field}.

Using an adapted frame field along $\g$, consider the following Jacobi equation in local coordinates
\[
	f''(s) + \k(s)f(s) = 0 \; , \tag{J}
\]
where $\k(s)$ is the Gaussian curvature. Then the solutions $f(s)$ to this equation correspond precisely to the perpendicular Jacobi fields $J(s) = f(s) E_1(s)$ along $\g$.

Next, recall the following useful definition from differential equations.
\begin{definition} Let $f(s)$ and $g(s)$ be differentiable. Then the \textit{Wronskian} $W(f,g)(s)$ is given by
\[
	W(f,g)(s) = f'(s)g(s) - f(s)g'(s) \; .
\]
\end{definition}

We have that for any two solutions $x, y$ of (J), the Wronskian is constant, for
\begin{align*}
	W(f,g)'(s) &= f''(s) g(s) - f(s) g''(s) \\
	&= -\k(s)f(s)g(s) + \k(s)f(s)g(s) \\
	&= 0 \; .
\end{align*}

Now we define some Jacobi fields that will be helpful soon. Let $a(s)$ be the unique solution to (J) such that
\[
	a(0) = 0 \text{ and } a'(0) = 1 \;.
\]
Since $M$ has no conjugate points, $a(s) \neq 0$ for all $s\neq 0$.  In particular, $a(s) < 0$ for $s < 0$ and $a(s) > 0$ for $s > 0$ by the above conditions. Next, for $t > 0$, let $d_t(s)$ be the unique solution to (J) such that
\[
	d_t(0) = 1 \text{ and } d_t(t) = 0 \; .
\]
	In fact, for $s > 0$, $d_t(s)$ is given by
\[
	d_t(s) = a(s) \int_s^t a(u)^{-2} du \; ,
\]
where the superscript denotes exponentiation, not function composition and inversion. To see this, let $b_t(s)$ be the integral expression above for $s > 0$. Then direct computation shows that $b_t(s)$ satisfies (J) and that $b_t(t) = 0$ and $b_t'(t) = -1/a(t)$. Thus for $s > 0$, $b_t(s)$ is the unique solution of (J) such that $b_t(t) = 0$ and $b_t'(t) = -1/a(t)$. Extend $b_t(s)$ to all $s\in\bb R$ using the existence and uniqueness theorem on (J). It is left to show that $b_t(0) = 1$, for then again by the existence and uniqueness theorem, we shall have $b_t(s) \equiv d_t(s)$ for all $s\in\bb R$. Indeed, we have $W(a, b_t)(t) = -a(t)b_t'(t) = 1$, and since the Wronskian is constant,
\[
	1 = W(a, b_t)(0) = b_t(0) \; ,
\]
as desired.

Next, for $0 < r < t$ and all $s > 0$, we have 
\[
	d_t(s) - d_r(s) = a(s) \int_r^t a(u)^{-2}du \;. 
\]
Notice that since $a(s) > 0$ for $s > 0$, $d_t(s)$ is strictly increasing in $t$ for fixed $s > 0$. By direct computation,
\[
	d_t'(0) - d_r'(0) = \int_r^t a(u)^{-2} du \; ,
\]
which is positive and strictly increasing in $t$. We are now ready for the following result, which will be useful for defining $X_s(v)$ and $X_u(v)$.

\begin{proposition}\label{Hao1} With the definitions from above, we have
\begin{enumerate}
	\item $\lim_{t\to\infty} d_t'(0)$ exists.
	\item Let $d(s)$ be the solution of (J) such that
\[
	d(0) = 1 \text{ and } d'(0) = \lim_{t\to\infty} d_t'(0)\; .
\]
		Then $d(s) = \lim_{t\to\infty} d_t(s)$.
	\item $d(s) \neq 0$ for all $s$.
\end{enumerate}
\end{proposition}
\begin{proof}
	This proof is an adaptation of Lemma 1 in \cite{GreenHopf}. For (1), it suffices to show that \linebreak$\lim_{t\to\infty} d_t'(0) - d_r'(0)$ exists for fixed $r$ with $0 < r < t$. Given $q > 0$, we have by direct computation that the solution $d_{-q}$ of (J) for which $d_{-q}(0) = 1$ and $d_{-q}(-q) = 0$ can be given by
\[
	d_{-q}(s) = a(s)n_{t, -q} + d_t(s)
\]
	for any $t > 0$, where
\[
	n_{t, -q} = -a(-q)^{-1} d_t(-q) \; .
\]
	Notice that the choice of $t$ does not matter by uniqueness of $d_{-q}$. Next, we claim that $n_{t, -q}$ is positive. To see this, we have $a(s) < 0$ when $s < 0$, as previously mentioned. Next, since $d_t(0) = 1$ and $d_t(t) = 0$, the fact that $M$ has no conjugate points implies $d_t(s) > 0$ for $s < t$, so $d_t(-q) > 0$. By direct computation, we have
\[
	d_{-1}'(0) - d_t'(0) = n_{t, -1} > 0
\]
	for any $t > 0$. Thus $d_t'(0) - d_r'(0)$ is increasing in $t$ and bounded above by $d_{-1}'(0) - d_r'(0)$, and hence the desired limit exists. The fact that $d(s) = \lim_{t\to\infty} d_t(s)$ follows from the fact that the solutions of (J) depend continuously on the initial conditions.
	
	Finally, we show that $d(s) \neq 0$ for all $s$. Assume toward contradiction that $d(t_0) = 0$ for some $t_0 \in\bb R$. Since $d(0) = 1$ by definition, we have that $t_0 \neq 0$, and by uniqueness of the solutions of (J), we have $d(s) \equiv d_{t_0}(s)$. Since $d_t(s)$ is strictly increasing in $t$ for $t > 0$ for fixed $s > 0$, we have $d_{t_0}(s) > d_t(s)$ for $t > 0$ and $s > 0$, which means $t_0$ cannot be positive. Thus $t_0 = -q_0$ for some $q_0 > 0$, so $d(s) \equiv d_{-q_0}(s)$. Notice that since $d_{-q}'(0) = n_{t, -q} + d_t'(0)$, we have $d_{-q}'(0)$ is continuous in $q$ for $q > 0$. Let $U$ be a positive interval containing $q_0$. By the existence and uniqueness theorem, we have that the map $q \maps d_{-q}'(0)$ sends $U$ homeomorphically onto an interval containing $d_{-q}'(0)$. Since $\lim_{t\to\infty} d_t'(0) = d'(0) = d_{-q_0}'(0)$, this implies that $d_t'(0) = d_{-q}'(0)$ for some $t > 0$ and $q > 0$, a contradiction by the existence and uniqueness theorem. Thus $d(s)$ is nonvanishing.
\end{proof}

\begin{remark}\label{dexplicit}
	For $s > 0$, $d(s)$ can be given by
\[
	d(s) = a(s) \int_s^\infty a(u)^{-2} du \;.
\]
\end{remark}

\begin{corollary}\label{Hao1} Again with the definitions from above, we have
\begin{enumerate}
	\item $\lim_{t\to -\infty} d_t'(0)$ exists.
	\item Let $\bar d(s)$ be the solution of (J) such that
\[
	\bar d(0) = 1 \text{ and } \bar d'(0) = \lim_{t\to -\infty} d_t'(0)\; .
\]
		Then $\bar d(s) = \lim_{t\to -\infty} d_t(s)$.
	\item $\bar d(s) \neq 0$ for all $s$.
\end{enumerate}
\end{corollary}
\begin{proof}
	Apply the results from the preceding proposition to the reverse geodesic $\g (-s)$.
\end{proof}

Next, for any number $t\neq 0$, we define a linear map $\xi \to \xi_t: T_v(TM) \to T_v(TM)$ for every $v\in TM$. Given a vector $v\in TM$ and a vector $\xi\in T_v(TM)$, let $J(s)$ be the unique Jacobi field along $\g_v$ such that
\[
	J(0) = d\pi (\xi) \text{ and } J(t) = 0 \; .
\]
	Let $\xi_t\in T_v(TM)$ be the vector that corresponds to $J$ under our isomorphism, i.e., we have $J = J_{\xi_t}$. Thus $\xi_t$ is the unique vector such that
\[
	 d\pi(\xi_t) = J_{\xi_t}(0) = d\pi(\xi) \text{ and } d\pi\circ dg^t(\xi_t) = J_{\xi_t}(t) = 0 \; .
\]
	The kernel of this map is $\ker d\pi$, which is vertical subspace of $T_v(TM)$.
	
	For $v\in SM$ it is not true in general that this map leaves $T_v(SM)$ invariant; in fact, for $\xi\in T_v(SM)$ we have $\xi_t\in T_v(SM)$ if and only if $\brak{\xi, Y(v)} = 0$. Indeed, if $\xi_t\in T_v(SM)$, then $s\to \brak{J_{\xi_t}(s), \g_v'(s)}$ is a constant function by Proposition \ref{P1.7} (3), but since $J_{\xi_t}(t) = 0$, we have $\brak{J_{\xi_t}(s), \g_v'(s)} \equiv 0$, and by Proposition \ref{P1.7} (5), $\brak{\xi, Y(v)} = 0$. Conversely, suppose $\brak{\xi, Y(v)} = 0$. Since $KY(v) = 0$ and $d\pi Y(v) = v$, we have
\[
	0 = \brak{\xi, Y(v)} = \brak{d\pi(\xi), v} = \brak{d\pi(\xi_t), v} = \brak{J_{\xi_t}(0), \g_v'(0)} \; .
\]

	Since $J_{\xi_t}(t) = 0$, we also have that it is perpendicular at $t$, and hence perpendicular everywhere. Thus again by Proposition \ref{P1.7} (3), $\xi_t\in T_v(SM)$.
	
\begin{definition}
	Let $v\in SM$. We define
\begin{enumerate}
	\item[(i)] $X_s(v) = \set{\xi\in T_v(SM)\text{ such that } \brak{\xi, Y(v)} = 0\text{ and } \xi_t\to\xi\text{ as } t\to\infty}$, and
	\item[(ii)] $X_u(v) = \set{\xi\in T_v(SM)\text{ such that } \brak{\xi, Y(v)} = 0\text{ and } \xi_t\to\xi\text{ as } t\to -\infty}$.
\end{enumerate}
	These are called the \emph{stable} and \emph{unstable subspaces} determined by $v$.
\end{definition}

\begin{definition}
	Let $\g$ be a unit speed geodesic in $M$ with initial velocity $v$. We define $J_s(\g)$ and $J_u(\g)$ to be the images in $J(\g)$ of $X_s(v)$ and $X_u(v)$, respectively, under our isomorphism. These are called the \emph{stable} and \emph{unstable subspaces} of perpendicular Jacobi fields along $\g$.
\end{definition}

\begin{remark}\label{HaoJacobi}
	(1) Let $\xi\in X_s(v)$. Since $J_{\xi_t}$ is perpendicular, we have that relative to an adapted frame field $E_1, E_2$ along $\g$ with $E_2(s) = \g'(s)$, $J_{\xi_t}$ can be given by
\[
	J_{\xi_t}(s) = cd_t(s) E_1(s)
\]
	for an appropriately chosen $c\in\bb R$ so that $d\pi\xi = cd_t(0)E_1(0) = cE_1(0)$. Notice that $c$ does not depend on $t$. Taking $t\to\infty$ shows that
\[
	J_{\xi}(s) = cd(s) E_1(s)\;.
\]
	For $\xi\in X_u(v)$, we have by similar reasoning that $J_{\xi}(s) = c\bar d(s) E_1(s)$. Therefore, we can express $J_s(\g)$ and $J_u(\g)$ as
\begin{enumerate}
	\item[(i)] $J_s(\g) = \set{cd(s) E_1(s): c\in\bb R}$ and
	\item[(ii)] $J_u(\g) = \set{c\bar d(s) E_1(s): c\in\bb R}$.
\end{enumerate}
	Clearly these are 1-dimensional vector subspaces of $J(\g)$, which means $X_s(v)$ and $X_u(v)$ are\linebreak 1-dimensional vector subspaces of $T_v(SM)$.
	
	(2) Let $\g$ be a unit speed geodesic in $M$. If $\k\equiv 0$, then $J_s(\g) = J_u(\g) = $ the vector space of all perpendicular parallel vector fields on $\g$. It will follow by Theorem \ref{T3.2} that for a compact surface $M$ without conjugate points the geodesic flow on $SM$ is of Anosov type if and only if $J_s(\g) \cap J_u(\g) = \set 0$ for every unit speed geodesic $\g$ in $M$.
\end{remark}

\begin{proposition}\label{R2.3.2}
	If $p: N\to M$ is a surjective local isometry of complete Riemannian manifolds and $M$ has no conjugate points, then $N$ has no conjugate points, and
\[
	dPX_s(v) = X_s(Pv) \text{ and } dPX_u(v) = X_u(Pv)
\]
	for every $v\in SN$, where $P = dp: TN \to TM$.
\end{proposition}
\begin{proof}
	$N$ has no conjugate points since for any Jacobi field $J$ on a geodesic $\g$ in $N$, $PJ$ is a Jacobi field on the geodesic $p\circ\g$ in $M$. If $\pi_1:TN\to N$ and $\pi_2:TM\to M$ are the projection maps, then $p\circ\pi_1 = \pi_2\circ P$. We also have $P\circ g^t = g^t\circ P$ for every $t\in\bb R$, where $g^t$ is the geodesic flow on both $TN$ and $TM$. From these two relations it follows that $(dP\xi)_t = dP(\xi_t)$ for any $t\neq 0$, any $v\in TN$, and any $\xi\in T_v(TN)$. If $\brak{\xi, Y(v)} = 0$, then $\brak{dP\xi, Y(Pv)} = 0$ by Proposition \ref{P1.5}. Therefore $dPX_s(v) = X_s(Pv)$ and $dPX_u(v) = X_u(Pv)$.
\end{proof}

\begin{proposition}\label{P2.4} We have
	\begin{enumerate}
		\item If $S: SM\to SM$ is the map which takes a vector $v$ to $-v$, then $X_u(-v) = dSX_s(v)$ and $X_s(-v) = dSX_u(v)$, and
		\item For any $t\in\bb R$ and any $v\in SM$, $dg^t X_s(v) = X_s(g^tv)$ and $dg^tX_u(v) = X_u(g^tv)$.
	\end{enumerate}
\end{proposition}

\begin{proof}[Proof of (1)]
	Let $\xi\in T_v(TM)$ and $t\neq 0$. It suffices to show that $dS(\xi_t) = (dS\xi)_{-t}$. By definition of $(dS\xi)_{-t}$, it suffices to show that $d\pi \circ dS(\xi_t) = d\pi\circ dS (\xi)$ and $d\pi\circ dg^{-t}\circ dS(\xi_t) = 0$. The first condition follows from the fact that $\pi\circ S = \pi$. The second condition also requires the fact that $S\circ g^a = g^{-a} \circ S$ for any $a\in\bb R$, which we have by
\[
	S\circ g^a(v) = S(\g_v'(a)) = -\g_v'(a) = \g_{-v}' (-a) = g^{-a}(-v) = g^{-a}\circ S(v) \; .
\]
\end{proof}
Before we prove (2), we need the following lemma.

\begin{lemma}\label{L2.5}
	Let $v\in SM$ be given. Then there exist numbers $b = b(v) > 0$ and $t_0 = t_0(v) > 0$ such that if $\xi\in T_v(TM)$ satisfies $d\pi\circ dg^t (\xi) = 0$ for some $t\geq t_0$, then $\magn{K\xi} \leq b\magn{d\pi\xi}$.
\end{lemma}

\begin{proof}
	We may assume that $d\pi\xi \neq 0$, or otherwise the Jacobi field $J_\xi \equiv 0$ since it vanishes at $0$ and $t\geq t_0 > 0$, and $K\xi = J_{\xi}'(0) = 0$. First, we consider the case when $J_\xi$ is perpendicular to $\g_v$. Relative to an adapted frame field along $\g_v$, if $d\pi\circ dg^t (\xi) = 0$ for some $t > 0$, we have $J_\xi = cd_t(s) E_1(s)$, with $c$ chosen so that $cd_t(0)E_1(0) = cE_1(0) = d\pi\xi$. Thus we have
\[
	\magn{K\xi} = \magn{J_\xi'(0)} = \abs{d_t'(0)}\abs{c} = \abs{d_t'(0)}\magn{d\pi \xi} \; .
\]
	Since $d_t'(0) \to d'(0)$ as $t\to\infty$, we can choose $t_0 = t_0(v) > 1$ so that $\abs{d_t'(0)} \leq 1 + \abs{d'(0)}$ for $t\geq t_0$. Setting $b = 1 + \abs{d'(0)}$, if $d\pi\circ dg^t(\xi) = 0$ for some $t\geq t_0$, then $\magn{K\xi} \leq b\magn{d\pi\xi}$.
	
	If $J$ is an arbitrary Jacobi field on $\g$, we may write $J(s) = J_1(s) + J_2(s)$, where $J_1$ is a perpendicular Jacobi field and $J_2$ is a tangential Jacobi field of the form $J_2(s) = (\a s + \b)\g'(s)$ for suitable constants $\a$ and $\b$. If $J(t) = 0$ for some $t\geq t_0$, then $J_1(t) = J_2(t) = 0$, and hence $\a = -\b/t$. Since $t \geq t_0 > 1$, we have $\magn{J_2'(0)} = \abs\a \leq \abs\b = \magn{J_2(0)}$ and $\magn{J_1'(0)} \leq b\magn{J_1(0)}$. Since $b \geq 1$, we have
\begin{align*}
	\magn{K\xi}^2 = \magn{J'(0)}^2 &= \magn{J_1'(0)}^2 + \magn{J_2'(0)}^2 \\
	&\leq b^2\magn{J_1(0)}^2 + b^2 \magn{J_2(0)}^2 \\
	&= b^2 \magn{J(0)}^2\;, 
\end{align*}
	and the result follows.
\end{proof}

\begin{proof}[Proof of (2)]
	 Let $v\in SM$ and let $\xi\in T_v(SM)$ be given, and fix a number $a\in\bb R$. For any $t\neq 0$, we have by the triangle inequality that $\magn{dg^a\xi - (dg^a\xi)_t} \leq \magn{dg^a(\xi -\xi_{t+a})} + \magn{dg^a(\xi_{t+a}) - (dg^a\xi)_t}$. If we let $\psi_t = dg^a(\xi_{t+a}) - (dg^a\xi)_t$, then $d\pi\circ dg^t \psi_t = 0$ and $d\pi\psi_t = d\pi\circ dg^a (\xi_{t+a} - \xi)$. If\linebreak $t\geq t_0 = t_0(g^av) > 0$, then by the previous lemma, $\magn{K\psi_t} \leq b\magn{d\pi\psi_t}$, where $b = b(g^a v) > 0$ is independent of $t$ and $\psi_t$. Therefore $\magn{\psi_t}^2 = \magn{K\psi_t}^2 + \magn{d\pi\psi_t}^2 \leq (1 + b^2)\magn{d\pi\circ dg^a(\xi_{t+a} -\xi)}^2 \leq (1 + b^2)\magn{dg^a(\xi_{t+a} -\xi)}^2$, and
\[
	\magn{dg^a\xi - (dg^a \xi)_t} \leq [1 + (1 + b^2)^{1/2}]\magn{dg^a(\xi_{t+a} - \xi)}
\]
	for $t\geq t_0$. If $\brak{\xi, Y(v)} = 0$, then $\brak{dg^a\xi, Y(g^av)} = 0$ by Proposition \ref{P1.7} (5). Since $a\in\bb R$ is fixed, $\xi\in X_s(v)$ if and only if $dg^a\xi\in X_s(g^av)$. This fact and (1) of this Proposition imply the similar invariance relation for $X_u$.
\end{proof}

\begin{proposition}\label{P2.7}
	Assume the Gaussian curvature $\k$ of $M$ is bounded from below, so $\k > -k^2$ for some $k > 0$, and let $J$ be a perpendicular Jacobi field on a unit speed geodesic $\g$ in $M$ such that $J(0) = 0$. Then $\magn{J'(s)} \leq k\coth (ks) \magn{J(s)}$ for every $s > 0$.
\end{proposition}
To show this, we first state and prove two lemmas, the first of which will also be important later.
\begin{lemma}\label{Green21}
	Let $r(s)$ be a real-valued function defined for all $s\in\bb R$ such that $r(s) \geq -k^2$ for all $s$ for some $k > 0$. If $u(s)$ is a solution of the Ricatti equation
\[
	u'(s) + u(s)^2 + r(s) = 0 \;, \tag{R}
\]
	defined everywhere, then $|u(s)| \leq k$.
\end{lemma}
\begin{proof}
	This proof is an adaptation of Lemma 2.1 in \cite{GreenConjugate}. Let $b>k$ and suppose there exists $s_0$ with $u(s_0) > b$. Then there is a solution $v(s)$ of
\[
	v'(s) = b^2 - v(s)^2
\]
	for which $v(s_0) = u(s_0)$. For a suitably chosen constant $d$, we have $v(s) = b\coth (bs - d)$. Subtracting the above by (R), we have
\[
	v'(s) - u'(s) = b^2 + r(s) - v(s)^2 + u(s)^2 \; .
\]
	For $s = s_0$, we have $v'(s_0) - u'(s_0) = b^2 + r(s) > 0$, so
\[
	v(s) \leq u(s)
\]
	in at least a left-sided neighborhood $s_0-\e \leq s\leq s_0$ of $s_0$. This, however, implies that the inequality holds for $s \leq s_0$, since for any $s_1$ where $v(s_1) = u(s_1)$, we have that $v'(s_1) - u'(s_1) > 0$, and so $v(s) \leq u(s)$ for a left-sided neighborhood of $s_1$. Now if we observe the behavior of $v(s)$, we notice that $\lim_{s \to d/b^+} v(s) = \infty$, so $\lim_{s\to d/b^+} u(s) = \infty$, but this contradicts the fact that $u(s)$ is defined everywhere. Since $b$ can be taken arbitrarily close to $k$, we have $u(s) \leq k$.
	
	For the lower bound, suppose there exists $s_0$ with $u(s_0) < -b$. We then consider $w(s) - u(s)$, where $w(s) = b\coth (bs - c)$ with $c$ chosen such that $w(s_0) = u(s_0)$. Using arguments similar to the above, we have $w'(s_0) - u'(s_0) > 0$, so $w(s) \geq u(s)$ for a right-sided neighborhood $s_0 \leq s \leq s_0 + \e$ of $s_0$, and hence for all $s \geq s_0$. Since $\lim_{s\to c/b^-} v(s) = -\infty$, we also have $\lim_{s\to c/b^-} u(s) = -\infty$, which again contradicts the fact that $u(s)$ is defined everywhere. Thus $u(s) \geq -k$.
\end{proof}

\begin{lemma}\label{L2.8}
	If $u(s)$ is a solution of the Ricatti equation (R) defined for all $s > 0$, then $\abs{u(s)} \leq k\coth (ks)$ for all $s > 0$.
\end{lemma}
\begin{proof}
	The first part of this proof is an adaptation of Lemma 3 in \cite{GreenHopf}. For fixed $s_0 > 0$, choose a number $d > 0$ such that $u(s_0) < k\coth (ks_0 - d)$. We show that $u(s) < k\coth(ks-d)$ for $s \geq s_0$. We have $v(s) = k\coth (ks-d)$ is a solution to the equation
\[
	v'(s) + v(s)^2 - k^2 = 0 \;,
\]
	defined for $s > d/k > 0$. If $f(s) = u(s) - v(s)$, then $f(s_0) < 0$. If $f(s) = 0$ for some $s > s_0$, let $s_1$ be the first such value. Differentiating at $s_1$, we obtain
\[
	f'(s_1) = v(s_1)^2 - u(s_1)^2 - r(s_1) - k^2 = -r(s_1) - k^2 < 0 \; .
\]
	This means that $f(s) > 0$ in some left-sided neighborhood $s_1 - \e < s < s_1$ of $s_1$, but since $s_0 < s_1$ and $f(s_0) < 0$, this contradicts the fact that $s_1$ is the first point after $s_0$ for which $f(s) = 0$. Thus $u(s) < k\coth(ks-d)$ for all $s \geq s_0$.

	Next, for each integer $n > 0$, choose a number $d_n > 0$ such that $u(1/n) < k\coth (k/n - d_n)$. If $s > 0$ is given, then $s > 1/n$ for sufficiently large $n$, and the above argument shows that $u(s) < k\coth (ks-d_n)$. Since $d_n < k/n$, $d_n\to 0$ and we have $u(s) \leq k\coth (ks)$. Finally, the argument in the last paragraph of the proof of Lemma \ref{Green21} works if $u(s)$ is defined only for $s > 0$; we assume $s_0 > 0$ and contradict the fact that $u(s)$ is defined for $s>0$, which then gives us that $u(s) \geq -k$ for any $s > 0$, completing the proof.
\end{proof}

\begin{proof}[Proof of Proposition \ref{P2.7}]
	Let $\g$ be a unit speed geodesic in $M$. Relative to an adapted frame field along $\g$, we consider the solution $a(s)$ to equation (J) defined earlier in this section. Let $J$ be a perpendicular Jacobi field on $\g$ such that $J(0) = 0$. We may assume $J'(0) \neq 0$, since otherwise there is nothing to prove, and by linearity it suffices to consider the case where $\magn{J'(0)} = 1$. Since $a(s)$ is the unique solution to (J) with $a(0) = 0$ and $a'(0) = 1$, we have that $J(s) = \pm a(s) E_1(s)$. If $u(s) = a'(s)a(s)^{-1}$, then $u(s)$ is defined for $s > 0$ and is a solution to the Ricatti equation
\[
	u'(s) + u(s)^2 + \k(s) = 0 \; .
\]
	Applying the lemma we just proved, we have
\[
	\magn{J'(s)} = \abs{a'(s)} = \abs{u(s)a(s)} = \abs{u(s)}\abs{a(s)} = \abs{u(s)}\magn{J(s)} \leq k\coth (ks) \magn{J(s)} \; .
\]
\end{proof}

\begin{proposition}\label{P2.9}
	Assume the Gaussian curvature $\k$ of $M$ is bounded from below, so $\k > -k^2$ for some $k > 0$. Let $\g$ be a unit speed geodesic in $M$. For any $R > 0$ we can find a number $T = T(R,\g) > 0$ such that $\magn{J(s)} \geq R\magn{J'(0)}$ for $s\geq T$, where $J$ is any perpendicular Jacobi field on $\g$ such that $J(0) = 0$.
\end{proposition}
\begin{proof}
	It suffices to consider the case where $\magn{J'(0)} = 1$. As before, we have $J(s) = \pm a(s) E_1(s)$ relative to an adapted frame field $E_1, E_2$ along $\g$ with $E_2(s) = \g'(s)$. Thus we want to show that $\abs{a(s)} \geq R$ for $s\geq T$.
	
	For any $s > 0$, let $m(s) = d'(0) - d_s'(0) = \displaystyle \int_s^\infty a(u)^{-2} du$, and let $u(s) = a'(s)a(s)^{-1}$ and\linebreak $v(s) = d'(s)d(s)^{-1}$. Then $u(s)$ and $v(s)$ are solutions defined for $s>0$ of the Ricatti equation $u'(s) + u(s)^2 + \k(s) = 0$. Since $d(s) = a(s)m(s)$, we have $d'(s) = a'(s)m(s) + a(s)m'(s) = a'(s)m(s) - a(s)^{-1}$, and consequently
\[
	u(s) - v(s) = a(s)^{-2} m(s)^{-1}  \; .
\]
	Choosing $s_0 > 0$ such that $k\coth ks \leq 2k$ for $s \geq s_0$, we have by Lemma \ref{L2.8} that
\[
	\abs{a(s)^{-2} m(s)^{-1}} \leq \abs{u(s)} + \abs{v(s)} \leq 4k 
\]
	for $s \geq s_0$, or in other words, that
\[
	\abs{a(s)} = 1/\abs{a(s)^{-1}} \geq (4k\abs{m(s)})^{-1/2}
\]
	for $s \geq s_0$. Let $R > 0$ be given. Since $m(s) \to 0$ as $s \to \infty$, we may choose $T > s_0 > 0$ such that $s\geq T$ implies $\abs{m(s)} \leq 1/(4kR^2)$, hence giving us $\abs{a(s)} \geq R$ for $s \geq T$.
\end{proof}

\begin{proposition}\label{P2.11}
	Assume that $\k > -k^2$ for some $k > 0$. Then for any $v\in SM$ and any $\xi\in X_s(v)$ or $X_u(v)$ we have $\magn{K\xi} \leq k\magn{d\pi\xi}$.
\end{proposition}
\begin{proof}
	First let $\xi\in X_s(v)$. Relative to an adapted frame field along $\g$, we have $\xi$ corresponds to $J_\xi = cd(s) E_1(s)$ for some $c\in\bb R$. Further, for $t\neq 0$, $\xi_t$ corresponds to $J_{\xi_t}(s) = cd_t(s) E_1(s)$. By reversing our geodesic and shifting by $t$, i.e., by considering the geodesic $\g(t-s)$ instead of $\g(s)$, we have the corresponding Jacobi field $L_{\xi_t}(s) = J_{\xi_t}(t-s) = d_t(t-s)E_1(t-s)$. In particular, we have $L_{\xi_t}(0) = J_{\xi_t}(t) = 0$, so by Proposition \ref{P2.7}, we have
\[
	\magn{L_{\xi_t}'(s)} \leq k\coth (ks) \magn{L_{\xi_t}(s)}
\]
	for $s > 0$. Setting $s = t$, we have
\[
	\magn{J_{\xi_t}'(0)} = \magn{L_{\xi_t}'(t)} \leq k\coth (kt) \magn{L_{\xi_t}(t)} = k\coth(kt) \magn{J_{\xi_t}(0)} \; .
\]
	By our isomorphism, we have
\[
	\magn{K\xi_t} \leq k\coth (kt) \magn{d\pi(\xi_t)} = k\coth (kt) \magn{d\pi\xi} \;.
\]
	Taking $t\to\infty$ yields the desired result. A similar approach shows the case where $\xi\in X_u(v)$.
\end{proof}

\begin{proposition}\label{P2.12}
	Assume that $\k > -k^2$ for some $k>0$. Let $v\in SM$ and let $\xi\in T_v(SM)$ be such that $\brak{\xi, Y(v)} = 0$ and $\magn{d\pi\circ dg^t\xi}$ is bounded above for all $t\geq 0$ (resp. for all $t\leq 0$). Then $\xi\in X_s(v)$ (resp. $\xi\in X_u(v)$).
\end{proposition}
\begin{proof}
	Suppose $\magn{d\pi\circ dg^t \xi} \leq A$ for some $A>0$. For any $t > 0$, $\xi_t\in T_v(SM)$ by Proposition \ref{P1.7} (3), since $J_{\xi_t}$ is perpendicular at 0 and $t$ and hence everywhere. Thus $\xi-\xi_t \in T_v(SM)$. By linearity, $J_{\xi-\xi_t} (0) = J_\xi(0) - J_{\xi_t}(0) = 0$ and $J_{\xi-\xi_t}(t) = J_\xi(t)$, which is perpendicular by assumption, so $J_{\xi-\xi_t}$ is also perpendicular everywhere. By Proposition \ref{P2.9}, for each integer $n > 0$ we have
\[
	\magn{\xi-\xi_t} = \magn{K(\xi-\xi_t)} = \magn{J_{\xi-\xi_t}'(0)} \leq \magn{J_{\xi-\xi_t}(s)}/n
\]
	for any $s\geq T(n, \g)$. Thus for $t \geq T(n, \g)$, 
\[
	\magn{\xi-\xi_{t}} \leq \magn{J_{\xi-\xi_{t}}(t)}/n = \magn{J_{\xi}(t)}/n = \magn{d\pi\circ dg^{t} \xi}/n \leq A/n \; ,
\]
	so $\xi_{t} \to \xi$ as $t\to\infty$, as desired. The proof for the other case works similarly.
\end{proof}

\begin{proposition}\label{P2.13}
	Assume that $\k > -k^2$ for some $k>0$. Suppose that there exist constants $A > 0$ and $s_0 \geq 0 $ so that for any perpendicular Jacobi vector field $J$ such that $J(0) = 0$ and for any numbers $t\geq s\geq s_0$ we have $\magn {J(t)} \geq A\magn{J(s)}$.  Then for any $v\in SM$ and any $\xi\in T_v(SM)$, $\xi\in X_s(v)$ (resp. $\xi\in X_u(v)$) if and only if $\brak{\xi,Y(v)} = 0$ and $\magn{d\pi\circ dg^t \xi}$ is bounded above for all $t \geq 0$ (resp. for all $t \leq 0$). In either case, $(1/A)\magn{d\pi\xi}$ is an upper bound.
\end{proposition}
\begin{proof}
	The reverse direction follows immediately from the preceding result. For the forward direction, let $v\in SM$ and $\xi\in X_s(v)$, and fix a number $u\geq 0$. Let $t \geq u + s_0$ and consider the Jacobi field $L_{\xi_t} (s) = J_{\xi_t}(t-s)$. Since $L_{\xi_t}(0) = 0$, we have
\[
	\magn{d\pi\circ dg^u (\xi_t)} = \magn{L_{\xi_t}(t-u)} \leq (1/A)\magn{L_{\xi_t}(t)} = (1/A) \magn{d\pi (\xi_t)} = (1/A)\magn{d\pi\xi} \; .
\]
	Since $\xi_t\to \xi$ as $t\to\infty$, we have
\[
	\magn{d\pi\circ dg^u \xi} \leq (1/A)\magn{d\pi\xi}
\]
	by continuity. The proof for the case where $\xi\in X_u(v)$ is similar.
\end{proof}

\begin{remark}\label{RConst}
	This result shows that assuming the hypothesis of Proposition \ref{P2.13} and given any unit speed geodesic $\g$ in $M$, a perpendicular Jacobi field $J$ lies in $J_s(\g)$ (resp. $J_u(\g)$) if and only if $\magn{J(t)} \leq (1/A)\magn{J(0)}$ for $t\geq 0$ (resp. $t\leq 0$). We have from Corollary \ref{HaoNoFocal} that the hypothesis is satisfied for $A=1$ and $s_0=0$.
\end{remark}

\begin{corollary}\label{C2.14}
	Let $M$ satisfy the hypothesis of Proposition \ref{P2.13} and let $B = [(1+k^2)/A^2]^{1/2}$. Then for any $v\in SM$ and any $\xi\in T_v(SM)$, $\xi \in X_s(v)$ (resp. $X_u(v)$) if and only if $\brak{\xi, Y(v)} = 0$ and $\magn{dg^t\xi}$ is bounded above for all $t\geq 0$ (resp. $t\leq 0$). In either case, $B\magn{\xi}$ is an upper bound.
\end{corollary}
\begin{proof}
	If $\magn{dg^t \xi}$ is bounded above for all $t\geq 0$ (resp. $t\leq 0$), then $\magn{d\pi\circ dg^t\xi} \leq \magn{dg^t\xi}$ is bounded above, and $\xi\in X_s(v)$ (resp. $X_u(v)$) by the previous proposition. Conversely, if $\xi\in X_s(v)$ (resp. $X_u(v)$), then by Propositions \ref{P2.4} (2), \ref{P2.11}, and \ref{P2.13} we have for $t\geq 0$ (resp. $t\leq 0$),
\[
	\magn{dg^t\xi}^2 = \magn{d\pi\circ dg^t\xi}^2 + \magn{K\circ dg^t\xi}^2 \leq (1 + k^2) \magn{d\pi\circ dg^t\xi}^2 \leq ((1+k^2)/A^2)\magn{d\pi\xi}^2 \leq B^2\magn{\xi}^2 \; .
\]
\end{proof}

Let $J_0^*$ be the set of all perpendicular Jacobi fields on unit speed geodesics of $M$.
\begin{proposition}\label{P2.15}
	Let the universal cover $H$ of $M$ be compactly homogenous. For each $s\geq 0$, let $g(s) = \inf \set{\magn{J(s)}: J\in J_0^*, J(0) = 0, \text{ and } \magn{J'(0)} = 1}$. Then $g(s) > 0$ for each $s > 0$, and $g$ is semi-continuous on $(0, \infty)$.
\end{proposition}
\begin{proof}
	Since $g$ is the infimum of a set of continuous functions, we have for any $\a > 0$, the set \linebreak $\set{s > 0: g(s) < \a}$ is open in $(0, \infty)$. Thus $g$ is semi-continuous. If $g^*(s)$ is the function defined in the same way relative to Jacobi fields in $H$, then $g^*(s) = g(s)$ since $p: H\to M$ is a surjective local isometry, and the Jacobi fields along geodesics in $M$ are the projections of Jacobi fields along geodesics in $H$. Let $L_0^*$ be the set of all perpendicular Jacobi fields on unit speed geodesics in $H$. To prove $g(s) = g^*(s)$ is positive for $s > 0$, it suffices to show for fixed $s > 0$ that $g^*(s) = \magn{L(s)}$ for some $L\in L_0^*$ such that $L(0) = 0$ and $\magn{L'(0)} = 1$.
	
	Let $s > 0$ be fixed, and let $(L_n) \sb L_0^*$ be a sequence of Jacobi fields with $L_n(0) = 0$ and $\magn{L_n'(0)} = 1$ such that $L_n(s) \to g^*(s)$ as $n\to\infty$. This corresponds to a sequence $(\xi_n)\sb T(SH)$ with $d\pi \xi_n = 0$ and $\magn{\xi_n} = 1$ with $\magn{d\pi\circ dg^s \xi_n} \to g^*(s)$ as $n\to\infty$. By Proposition \ref{P1.12}, there exist a sequence $\vp_n$ of isometries of $H$ and a vector $\xi^* \in T(SH)$ such that $\xi_n^* = dT_{\vp_n} \xi_n \to \xi^*$ by passing to a subsequence, where $T_{\vp_n} = d\vp_n : TH\to TH$. Since $T_{\vp_n}$ is an isometry, we have by the relations shown in the proof of Lemma \ref{P1.5} that $d\pi\xi_n^* = 0$, $\magn{\xi_n^*} = 1$, and
\[
	d\pi\circ dg^s \xi_n^* = d\pi\circ dg^s \circ dT_{\vp_n} (\xi_n) = T_{\vp_n} \circ d\pi \circ dg^s (\xi_n) \; .
\]
	Therefore,
\[
	\magn{d\pi\circ dg^s \xi_n^*} = \magn{d\pi\circ dg^s \xi_n} \to g^*(s) \; .
\]
	If $L$ is the Jacobi field corresponding to $\xi^*$, then $g^*(s) = \magn{L(s)}$.
\end{proof}

\newpage
\section{Anosov equivalences}

We are finally able to state and prove the main result of this paper. First we need the definition of an Anosov flow.

\begin{definition}
	Let $\vp^t$ be a complete $C^\infty$ flow on a (complete) $C^\infty$ Riemannian manifold $N$ of dimension $n\geq 3$. Then the flow is said to be \emph{Anosov} or of \emph{Anosov type} if the following conditions are satisfied:
\begin{enumerate}
	\item The vector field $Y = \left.\pd {\vp^t}{t}\right\rvert_{t=0}$ defined by the flow never vanishes on $N$.
	\item For each $n\in N$ the tangent space $T_nN$ splits in to a direct sum
	\[
		T_nN = X_s^*(n) \oplus X_u^*(n)\oplus Z(n)
	\]
		($\dim X_s^* = k > 0, \dim X_u^* = l > 0, \dim Z = 1$), where $Z(n)$ is generated by $Y(n)$, i.e., $Z(n) = \set{cY(n) : c\in\bb R}$, and there exist positive number $a$, $b$, and $c$ such that
		\begin{enumerate}
			\item[(i)] for any $\xi \in X_s^*(n)$,
			\[
				\magn{d\vp^t \xi}\leq a\magn\xi e^{-ct} \text{ for } t\geq 0\text{ and } \magn{d\vp^t \xi}\geq b\magn\xi e^{-ct} \text{ for } t\leq 0 \; , \text{ and}
			\]
			\item[(ii)] for any $\eta \in X_u^*(n)$,
			\[
				\magn{d\vp^t \eta}\leq a\magn\eta e^{ct} \text{ for } t\leq 0\text{ and } \magn{d\vp^t \eta}\geq b\magn\eta e^{ct} \text{ for } t\geq 0 \; .
			\]
		\end{enumerate}
\end{enumerate}
\end{definition}

We now state the main theorem. For the rest of this section we assume that $M$ is a complete, $C^\infty$ Riemannian manifold of dimension two without conjugate points such that the universal cover $H$ is compactly homogenous. Let $g^t$ denote the geodesic flow on $TM$, $SM$, $TH$, and $SH$. Since $M$ has the same values for the Gaussian  curvature as $H$, we have that $\k$ is bounded, so there exists a number $k > 0$ such that $\k > -k^2$ in $M$.

\begin{theorem}\label{T3.2} Let $M$ be a surface without conjugate points such that universal cover $H$ of $M$ is compactly homogenous. Then the following properties are equivalent:
	\begin{enumerate}
		\item The geodesic flow on $SM$ is of Anosov type.
		\item For every $v\in SM$, $X_s(v) \cap X_u(v) = \set{0}$.
		\item For every $v\in SM$, $T_v(SM) = X_s(v) \oplus X_u(v) \oplus Z(v)$, where $Z(v)$ is the one-dimensional subspace generated by $Y(v)$.
		\item There exists no nonzero perpendicular Jacobi field $J$ on a unit speed geodesic $\g$ of $M$ such that $\magn{J(t)}$ is bounded for all $t\in\bb R$.
	\end{enumerate}
	
	If, further, $M$ has no focal points, the above conditions are also equivalent to the following:
	\begin{enumerate}
		\setcounter{enumi}{4}
		\item There exists no nonzero perpendicular parallel Jacobi field $J$ on a unit speed geodesic $\g$ of $M$.
		\item Every geodesic of $M$ passes through a point of negative Gaussian curvature.
	\end{enumerate}
\end{theorem}

We shall prove this theorem in multiple steps, which will require intermediary results.

\begin{proposition}\label{P3.7} Assume that $M$ admits no nonzero perpendicular Jacobi field $J$ on a unit speed geodesic $\g$ such that $\magn{J(t)}$ is bounded above for all $t\in\bb R$. Then there exists a constant $A > 0$ such that if $J$ is a nonzero perpendicular Jacobi field on a unit speed geodesic $\g$ such that $J(0) = 0$, then $\magn{J(t)} \geq A\magn{J(s)}$ for any numbers $t\geq s\geq 1$.
\end{proposition}

\begin{proof} 
	It suffices to prove this result for the universal cover $H$ of $M$. Indeed, the covering map is a surjective local isometry, so Jacobi fields along a geodesic in $M$ are projections of Jacobi fields along geodesics in $H$.
	
	Assume toward contradiction that the proposition is false. Then there exist nonzero perpendicular Jacobi fields $J_n$ on unit speed geodesics $\g_n$ in $H$, along with sequences $1 \leq s_n \leq t_n$ such that for each integer $n > 0$, we have $J_n(0) = 0$ and $\magn{J_n(t_n)} \leq (1/n) \magn{J_n(s_n)}$. By linearity, we may assume $\magn{J_n'(0)} = 1$.
	
	By the extreme value theorem, we can choose a sequence $(u_n) \sb \bb R$ such that $0 \leq u_n \leq t_n$ and $\magn{J_n(s)} \leq \magn{J_n(u_n)}$ for every $0\leq s\leq t_n$. We claim that the $u_n$ are lower bounded by some $\d > 0$. For if this were not the case, then $u_n \to 0$ by passing to a subsequence. By Proposition \ref{P1.11}, there exists a sequence $\vp_n$ of isometries of $H$ and a Jacobi field $L$ on a unit speed geodesic $\g$ in $H$ such that $L_n = d\vp_n J_n \to L$ by passing to a subsequence. By continuity,
\[
	\lim_{n\to\infty} \magn{J_n(u_n)} = \lim_{n\to\infty} \magn{L_n(u_n)} = \magn{L(0)} = 0 \; .
\]
	On the other hand, since $t_n \geq 1$ for all $n$, we have by Proposition \ref{P2.15} that
\[
	\magn{J_n(u_n)} \geq \magn{J_n(1)} \geq g(1) > 0 \; ,
\]
 	a contradiction.
 	
 	Next, for each $n$, let $\s_n(t) = \g_n(u_n + t)$ and $L_n(t) = J_n(u_n + t)/\magn{J_n(u_n)}$. Then $L_n$ is a Jacobi field along $\s_n$ with $\magn{L_n(0)} = 1$, $\magn{L_n(s)} \leq 1$ for $-u_n\leq s\leq t_n-u_n$, $L_n(-u_n) = 0$, and $\magn{L_n(t_n-u_n)} \leq 1/n$. Since $H$ is compactly homogenous, $\k > -k^2$ for some $k > 0$, and by Proposition \ref{P2.7}, $\magn{L_n'(0)} = \magn{J_n'(u_n)}/\magn{J_n(u_n)} \leq k\coth (ku_n)\leq k\coth (k\d)$. By Proposition \ref{P1.11}, there exists a sequence $(\vp_n)$ of isometries of $H$ and a Jacobi field $L^*$ on a unit speed geodesic $\g$ such that $L_n^* = d\vp_n L_n \to L^*$ by passing to a subsequence. Since $\magn{L^*(0)} = 1$ by continuity, we have $L^*\not\equiv 0$.
 	
 	By passing to a subsequence, we now have four cases for $\lim_{n\to\infty} t_n-u_n$ and $\lim_{n\to\infty} u_n$, namely each limit can either be finite and nonnegative or be positive infinity. We shall derive contradictions in each of these cases.
 	
 	First, suppose $t_n-u_n \to l$ and $u_n \to u$ for $l\geq 0$ and $u\geq \d> 0$. Continuity yields $L^*(-u) = 0$ and $L^*(l) = 0$. But $t\geq 0$ and $-u \leq -\d$, and since $H$ has no conjugate points, we have $L^*\equiv 0$, a contradiction.
 	
 	Next, suppose $u_n\to\infty$ and $t_n-u_n \to l$ for $l\geq 0$. By continuity, $L^*(l) = 0$ and $\magn{L^*(u)} \leq 1$ for all $u\leq l$. By Proposition \ref{P2.9} applied to $L^*(l-t)$, however, $\magn{L^*(t)}\to\infty$ as $t\to -\infty$, a contradiction. The case when $t_n-u_n\to\infty$ and $u_n\to u$ for $u\geq \d> 0$ is proved similarly.
 	
 	Finally, when both sequences approach infinity, then $\magn{L^*(t)}\leq 1$ for all $t\in\bb R$ by continuity, but this contradicts the assumption of the proposition.
\end{proof}
 	
\begin{proof}[Proof of (2) $\iff$ (3) $\iff$ (4)]
 	For any $v\in SM$, we have $T_v(SM) = Z(v)^{\perp} \oplus Z(v)$, where $Z(v)$ is the 1-dimensional subspace generated by $Y(v)$. Since $SM$ is 3-dimensional, $Z(v)^{\perp}$ is 2-dimensional. Since $X_s(v)$ and $X_u(v)$ are 1-dimensional subspaces of $Z(v)^{\perp}$, we immediately have (2) $\iff$ (3).
 	
 	Next, suppose that $M$ admits a nonzero perpendicular Jacobi field $J$ on a unit speed geodesic $\g$ such that $\magn{J(t)}$ is bounded everywhere. If $v = \g'(0)$ and $\xi\in T_v(SM)$ corresponds to $J$, then Proposition \ref{P2.12} shows that $\xi\in X_s(v) \cap X_u(v)$, and thus (2) $\Longrightarrow$ (4). The reverse direction follows from Propositions \ref{P2.13} and \ref{P3.7}.
\end{proof}

\begin{proposition}\label{P3.9}
	Assume that $M$ admits no nonzero perpendicular Jacobi field $J$ on a unit speed geodesic $\g$ such that $\magn{J(t)}$ is bounded above for all $t\in \bb R$. Then for any number $\e > 0$, there exists a number $T>0$ such that 
	\begin{enumerate}
		\item for any $v\in SM$ and any $\xi\in X_s(v)$, $\magn{dg^t\xi} \leq \e\magn{\xi}$ for $t\geq T$, and
		\item for any $v\in SM$ and any $\eta\in X_u(v)$, $\magn{dg^t\eta} \leq \e\magn{\eta}$ for $t\leq - T$.
	\end{enumerate}
\end{proposition}
\begin{proof}
	It suffices by Propositions \ref{P1.5} and \ref{R2.3.2} to prove these assertions for the universal cover $H$ of $M$. We start with the first assertion, and assume toward contradiction that there exists an $\e > 0$ for which there are sequences $(t_n) \sb\bb R$, $(v_n)\sb SH$, and $(\xi_n)\sb T(SH)$ such that $\xi_n\in X_s(v_n)$, $t_n\to\infty$, and $\magn{dg^{t_n}\xi_n} > \e \magn{\xi_n}$ for every integer $n$. By linearity, we may assume $\magn{\xi_n} = 1$. If $\psi_n = dg^{t_n} \xi_n$, then $\psi_n\in A_s = \bigcup_{v\in SH} X_s(v)$ and $\magn{\psi_n} > \e$ for every $n$. Corollary \ref{C2.14} and Proposition \ref{P3.7} then imply that there exists a number $B>0$ such that for every $n$, $\magn{dg^t\psi_n} \leq B$ for $-t_n\leq t < \infty$. By Proposition \ref{P1.12}, there exist a sequence $\vp_n$ of isometries of $H$ and a vector $\psi^* \in T(SH)$ such that $\psi_n^* = dT_{\vp_n}\psi_n\to\psi^*$ by passing to a subsequence, where $T_{\vp_n} = d\vp_n$. By Proposition \ref{R2.3.2}, we have $\psi_n^* \in A_s$ for every $n$. Since $g^t\circ d\vp_n = d\vp_n\circ g^t$ for all $t\in\bb R$, we have
\begin{align*}
	\magn{dg^t\psi_n^*} = \magn{dg^t\circ dT_{\vp_n} \psi_n} = \magn{dT_{\vp_n}\circ dg^t\psi_n} = \magn{dg^t \psi_n} \leq B\;
\end{align*}
	for $-t_n\leq t< \infty$, where the last equality follows because $\vp_n$ is an isometry. The vector $\psi^*\neq 0$ since $\magn{\psi_n^*} = \magn{\psi_n} > \e$ for every $n$, and by continuity $\magn{dg^t \psi^*} \leq B$ for all $t\in\bb R$. If $\psi^* \in T_{v^*}(SH)$ and $\psi_n^*\in T_{v_n^*} (SH)$, then $\brak{\psi^*,Y(v^*)} = 0$ since $\brak{\psi_n^*, Y(v_n^*)} = 0$ for every $n$. Thus $J_{\psi^*}$ is a nonzero perpendicular Jacobi field on $\g_{v^*}$ in $H$ whose magnitude is uniformly bounded by $B$ for all $t\in\bb R$. Then $J(t) = dp J_{\psi^*}(t)$ is a nonzero perpendicular Jacobi field on the geodesic $p\circ \g_{v^*}$ in $M$ such that $\magn{J(t)} = \magn{J_{\psi^*(t)}} \leq B$ for all $t\in\bb R$, contradicting the assumption in the proposition.
	
	For the second assertion, we simply note that $\eta\in X_u(v)$ if and only if $\xi=dS(\eta)\in X_s(-v)$, where $S: SM\to SM$ maps $v \maps -v$. Since $S$ is an isometry of $SM$ satisfying the relation $S\circ g^t = g^{-t}\circ S$ for all $t\in\bb R$, we have that $\magn{dg^{-t}\eta} = \magn{dS\circ dg^t \xi} = \magn{dg^t\xi} \leq \e\magn{\xi} = \e\magn{\eta}$ for all $t \geq T$.
\end{proof}

\begin{lemma}\label{L3.11}
	Let $M$ be a complete manifold without conjugate points, which satisfies the hypothesis of Proposition \ref{P2.13}. For each $s\geq 0$, let $\vp(s) = \sup\set{\magn{dg^s\xi}:\xi\in A_s, \magn{\xi} = 1}$. Then
	\begin{enumerate}
		\item there exists a constant $B > 0$ such that $0\leq \vp(s)\leq B$ for all $s\geq 0$, and
		\item $\vp(s+t)\leq \vp(s)\cdot\vp(t)$ for all $s\geq 0$ and $t\geq 0$. 
	\end{enumerate}
	Furthermore, if the universal cover $H$ is compactly homogenous and $M$ admits no nonzero perpendicular Jacobi field $J$ on a unit speed geodesic $\g$ such that $\magn{J(t)}$ is bounded above for all $t\in \bb R$, then
	\begin{enumerate}
		\item[(3)] $\vp(s) \to 0$ as $s\to \infty$.
	\end{enumerate}
\end{lemma}
\begin{proof}
	First, if $B>0$ is the constant in Corollary \ref{C2.14}, then $0\leq\vp(s)\leq B$ for $s\geq 0$.
	
	Next, for any $\xi\in A_s$ and any $s\geq 0$, $\magn{dg^s\xi}\leq \vp(s)\magn{\xi}$ by definition. Let $\xi\in A_s$, $\magn{\xi} = 1$ be given, and let $s\geq 0$, $t\geq 0$ be arbitrary. Then
\[
	\magn{dg^{s+t} \xi} = \magn{dg^s \circ dg^t \xi} \leq \vp(s) \magn{dg^t\xi} \leq \vp(s)\cd\vp(t) \; ,
\]
	since $A_s$ is invariant under $dg^t$. Taking the above over all $\xi\in A_s$ yields the desired result.
	
	The final property follows immediately from Proposition \ref{P3.9}.
\end{proof}

\begin{lemma}\label{L3.12}
	Let $\vp: (0,\infty)\to(0,\infty)$ be a function satisfying all three properties of the previous lemma. Then there exist numbers $a>0$ and $c>0$ such that $\vp(s) \leq ae^{-cs}$ for $s\geq 0$.
\end{lemma}
\begin{proof}
	Notice that the second property implies that if $s\geq 0$ is any number and $n>0$ is any integer, then $\vp(ns)\leq \vp(s)^n$. The third property tells us we can choose $s_0 > 0$ so that $\vp(s) < 1/2$ for $s\geq s_0$. Thus it suffices to show that $\vp(s)\leq e^{-cs}$ for $s\geq s_0$, where $c = \f 12 (\log 2)/s_0$. Given $s\geq s_0$, we may choose an integer $n\geq 1$ such that $s_0\leq s/n\leq 2s_0$. Let $s^* = s/n$. Then
\[
	\f {\log \vp(s)}s = \f{\log\vp(ns^*)}{ns^*} \leq \f{\log\vp(s^*)^n}{ns^*} = \f{\log\vp(s^*)}{s^*} \leq -\f{\log 2}{2s_0} = -c \; .
\]
\end{proof}

\begin{proof}[Proof of (1) $\iff$ (4).]
	We start with the reverse direction. Proposition \ref{P3.7}, Lemma \ref{L3.11}, and Lemma \ref{L3.12} together imply that for every $v\in SM$ and every $\xi\in X_s(v)$, we have
\[
	\magn{dg^t \xi} \leq a\magn\xi e^{-ct}
\]
	for $t\geq 0$, where $a>0$ and $c>0$ do not depend on $\xi$ or $v$. By replacing $\xi$ with $dg^{-t}\xi$, we have that
\[
	\magn{dg^{-t}\xi} \geq (1/a) \magn{\xi}e^{ct}
\]
	for $t \geq 0$. Now let $\eta \in X_u(v)$, so $\eta = dS(\xi)$ for $\xi\in X_s(-v)$. Since the first inequality is independent of $v$, we have
\[
	\magn{dg^t \eta} = \magn{dg^t \circ dS (\xi)} = \magn{dS\circ dg^{-t} \xi} = \magn{dg^{-t}\xi} \geq (1/a) \magn{\xi}e^{ct} = (1/a) \magn{\eta} e^{ct}
\]
	for $t \geq 0$. Since $(1/a) \magn{\eta}e^{ct} > a\magn\eta e^{-ct}$ for sufficiently large $t$, we have $\eta \notin X_s(v)$ for any $\eta\in X_u(v)$, so $X_s(v) \cap X_u(v) = \set 0$, and by using (2) $\Longrightarrow$ (3) we have that
\[
	T_v(SM) = X_s(v) \oplus X_u(v) \oplus Z(v) \; .
\]
	Setting $X_s^*(v) = X_s(v)$ and $X_u^*(v) = X_u(v)$, the conditions for an Anosov flow are satisfied.
	
	Next, we show the forward direction. Suppose toward contradiction that there exists a nonzero perpendicular Jacobi field $J$ on a unit speed geodesic such that $\magn{J(t)}\leq c$ for some $c>0$ and all $t\in\bb R$. If $v = \g'(0)$, then let $\xi\in T_v(SM)$ be such that $J = J_\xi$. Choose $k > 0$ such that $\k > -k^2$. Since $J$ is perpendicular, $\brak{\xi, Y(v)} = 0$, and $\xi \in X_s(v) \cap X_u(v)$ by Proposition \ref{P2.12}. For any $t\in \bb R$, we have $dg^t\xi \in X_s(g^t v) \cap X_u(g^t v)$ and
\[
	\magn{K\circ dg^t\xi} \leq k\magn{d\pi\circ dg^t\xi} \leq kc
\]
	by Proposition \ref{P2.11}. Thus by definition of the Sasaki metric, $\magn{dg^t\xi} \leq c(1+k^2)^{1/2}$ for all $t$. By definition of an Anosov flow, we have $\xi = \xi_1 + \xi_2 + \xi_3$, where $\xi_1\in X_s^*(v)$, $\xi_2\in X_u^*(v)$, and $\xi_3 = aY(v)$ for some $a\in\bb R$. Notice that $\magn{dg^t\xi_3} = \magn{\xi_3}$ for all $t$, since $dg^tY(w) = Y(g^tw)$ and $\magn{Y(w)} = 1$ for $w\in SM$. By the properties of the spaces $X_s^*(v)$ and $X_u^*(v)$, it follows that if $\xi_1\neq 0$, then $\magn{dg^t\xi} \to \infty$ as $t\to-\infty$, while if $\xi_2\neq 0$, then $\magn{dg^t\xi}\to\infty$ as $t\to\infty$. Thus $\xi_1 = \xi_2 = 0$, which contradicts the assumption that $\xi$ is nonzero and orthogonal to $Y(v)$.
\end{proof}

For the rest of the proof, we assume that $M$ has no focal points.
\begin{proof}[Proof of (4) $\iff$ (5).]
	We start with the forward direction and show the contrapositive. Suppose there exists a nonzero perpendicular parallel Jacobi field $J$ on a unit speed geodesic $\g$. Since $J'(t) = 0$, we have $\magn{J(t)} > 0$ is constant over all $t$. Then $J$ is a nonzero perpendicular Jacobi field whose magnitude is uniformly bounded for all $t$.
	
	Before we show the reverse direction, we need to make some observations. First, we claim that if a Jacobi field $J$ is in $J_s(\g)$ (resp. $J_u(\g)$), then $\magn{J(t)}$ is nonincreasing (resp. nondecreasing) in $t$. For let $J\in J_s(\g)$ and $a < b$ be given. By Remark \ref{RConst}, the condition of Proposition \ref{P2.13} is satisfied relative to the constants $A=1$ and $s_0=0$. Thus we have $\magn{J(t)} \leq \magn{J(0)}$ for $t\geq 0$ by Proposition \ref{P2.13}. For each $u\in\bb R$, let $\g^*(u) = \g(a+u)$ and $J^*(u) = J(a + u)$, effectively shifting the parameter of our geodesic and Jacobi field by $a$. Since $\magn{J^*(u)}\leq \magn{J(0)}$ for $u\geq -a$, Proposition \ref{P2.12} shows that $J^* \in J_s(\g^*)$. Then again using Proposition \ref{P2.13}, this time on $J^*$, we have that $\magn{J^*(t)}\leq \magn{J^*(0)}$ for $t\geq 0$. This yields
\[
	\magn{J(b)} = \magn{J^*(b-a)} \leq \magn{J^*(0)} = \magn{J(a)} \; ,
\]
	as desired. A similar argument shows the result for $J\in J_u (\g)$.
	
	We now prove the reverse direction by showing the contrapositive. Suppose that $J$ is a nonzero perpendicular Jacobi field on a unit speed geodesic $\g$ such that $\magn{J(t)}$ is uniformly bounded for all $t\in\bb R$. It follows from Proposition \ref{P2.12} that $J\in J_s(\g)\cap J_u(\g)$. The claim above shows that $\magn{J(t)}$ is both nonincreasing and nondecreasing in $t$, and hence constant. Since $J$ perpendicular, relative an adapted frame field $E_1, E_2$ along $\g$ with $E_2(s) = \g'(s)$, we have
\[
	J(t) = f(t) E_1(t)
\]
	for some real-valued function $f$, which then implies $f(t)$ is constant. Since $J'(t) = f'(t) E_1(t) \equiv 0$, we conclude $J$ is parallel.
\end{proof}

\begin{proof}[Proof of (1) $\iff$ (6)]
	For the forward direction, we show the contrapositive. Let $\g$ be a unit speed geodesic such that $\k(t) \geq 0$ along $\g$ for all $t$. Relative to an adapted frame field along $\g$, it suffices to show that $E_1$ is a Jacobi field, for then $E_1$ is a nonzero perpendicular parallel Jacobi field along $\g$, which by (1) $\Longrightarrow$ (5) implies the geodesic flow on $SM$ is not Anosov.
	
	We first show that $\k(t) \equiv 0$ on $\g$. Recall that $u(s) = d'(s)d^{-1}(s)$ is a solution of the Ricatti equation
\[
	u'(s) + u(s)^2 + \k(s) = 0
\]
	along $\g$. Since the Gaussian curvature is nonnegative, given any $\e > 0$, $\k(s) > -\e^2$ for all $s\in\bb R$. Since $u(s)$ is defined everywhere, we have $\abs{u(s)}\leq \e$ by Lemma \ref{Green21}, so $u(s) \equiv 0$, which then gives us $\k(s) \equiv 0$ using the Ricatti equation. This in turn implies that the curvature tensor is identically zero along $\g$, so the Jacobi equation along $\g$ simplifies to
\[
	J''(s) = 0 \; .
\]
	Since $E_1(s)$ is parallel, $E_1''(s) \equiv 0$, so $E_1$ is a Jacobi field, as desired.
	
	For the reverse direction, suppose every geodesic in $M$ passes through a point of negative Gaussian curvature. Assume toward contradiction that the geodesic flow on $SM$ is not of Anosov type, so (5) $\Longrightarrow$ (1) implies there exists a nonzero perpendicular parallel Jacobi field $J$ along $\g$. Relative to an adapted frame field along $\g$ and using the fact that $J \in J_s(\g)$, we have by Remark \ref{HaoJacobi} that $J(s) = cd(s)E_1(s)$ for some $c \neq 0$. Without loss of generality, we may assume $c = 1$. Since $J$ is parallel, $d(s)$ is constant, so $d'(s) \equiv 0$ along $\g$. Plugging $u(s) = d'(s) d^{-1}(s) \equiv 0$ into the Ricatti equation shows that $\k(s) \equiv 0$, a contradiction.
\end{proof}

\bibliographystyle{alpha}
\bibliography{paper.bib}

\end{document}